\documentclass[12pt,leqno]{amsart}
\usepackage{latexsym,amsmath,amssymb}
\usepackage{graphicx}

\title[Unrectifiability of metric spaces]{On conditions for unrectifiability of a metric space}

\author{Piotr Haj{\l}asz and  Soheil Malekzadeh}

\address{P.\ Haj{\l}asz: Department of Mathematics, University of Pittsburgh, 301
  Thackeray Hall, Pittsburgh, PA 15260, USA, {\tt hajlasz@pitt.edu}}

\address{S. Malekzadeh: Department of Mathematics, University of Pittsburgh, 301
  Thackeray Hall, Pittsburgh, PA 15260, USA, {\tt som13@pitt.edu}}

\thanks{P.H.\ was supported by NSF grant DMS-1161425.}

plus 6pt minus 12pt
\def\rank{{\rm rank\,}}

\def\eps{\varepsilon}

\def\vi{\varphi}

\def\H{{\mathcal H}}

\newtheorem{theorem}{Theorem}
\newtheorem{lemma}[theorem]{Lemma}

\newtheorem{proposition}[theorem]{Proposition}

\def\diam{{\rm diam\,}}
\def\dist{{\rm dist\,}}

\def\lip{{\rm Lip\,}}

\theoremstyle{definition}
\newtheorem{remark}[theorem]{Remark}
\newtheorem{definition}[theorem]{Definition}

\newcommand{\barint}{
\rule[.036in]{.12in}{.009in}\kern-.16in \displaystyle\int }

\newcommand{\barcal}{\mbox{$ \rule[.036in]{.11in}{.007in}\kern-.128in\int $}}
\newcommand{\bbbn}{\mathbb N}
\newcommand{\bbbr}{\mathbb R}

\newcommand{\Heis}{\mathbb H}

\def\ap{\operatorname{ap}}

\def\diam{\operatorname{diam}}

\def\dist{\operatorname{dist}}

\def\mvint_#1{\mathchoice
          {\mathop{\vrule width 6pt height 3 pt depth -2.5pt
                  \kern -8pt \intop}\nolimits_{\kern -3pt #1}}%
          {\mathop{\vrule width 5pt height 3 pt depth -2.6pt
                  \kern -6pt \intop}\nolimits_{#1}}%
          {\mathop{\vrule width 5pt height 3 pt depth -2.6pt
                  \kern -6pt \intop}\nolimits_{#1}}%
          {\mathop{\vrule width 5pt height 3 pt depth -2.6pt
                  \kern -6pt \intop}\nolimits_{#1}}}

\numberwithin{theorem}{section} \numberwithin{equation}{section}

\begin{document}

\sloppy

%\dsp

\begin{abstract}
We find necessary and sufficient conditions for a Lipschitz map
$f:\bbbr^k\supset E\to X$ into a metric space to satisfy
$\H^k(f(E))=0$. An interesting feature of our approach is that
despite the fact that we are dealing with arbitrary metric spaces,
we employ a variant of the classical implicit function theorem.
Applications include pure unrectifiability of the Heisenberg groups
and that of more general Carnot-Carath\'eodory spaces.
\end{abstract}

\maketitle

\section{Introduction}
\label{introduction}

We say that a metric space $(X,d)$ is {\em countably $k$-rectifiable}
if there is a family of Lipschitz mappings $f_i:\bbbr^k\supset E_i\to X$
defined on measurable sets $E_i\subset\bbbr^k$ such that
$$
\H^k\left(X\setminus\bigcup_{i=1}^\infty f_i(E_i)\right)=0.
$$
A metric space $(X,d)$ is said to be {\em purely $k$-unrectifiable}
if for any Lipschitz mapping $f:\bbbr^k\supset E\to X$,
where $E\subset\bbbr^k$ is measurable we have $\H^k(f(E))=0$.

The theory of rectifiable sets plays a significant role in geometric measure theory and calculus of variations. 
See e.g. \cite{federer, Mattila} for results in Euclidean spaces. 
Recent development of analysis on metric spaces extended this theory to metric spaces. See e.g.
\cite{ambrosiok,ambrosiok2,DS,kirchheim} and references therein.
Considering the importance of this theory, it is reasonable to search for simple geometric conditions 
which would guarantee that the image of a Lipschitz mapping from a subset of a Euclidean space into a metric spaces would have measure zero. 
One of the main results of this paper (Theorem~\ref{main}) establishes such conditions.

Let $f:Z\to (X,d)$ be a mapping between metric spaces and let $\{y_1,\ldots,y_k\}\subset X$ be given.
The mapping $g:Z\to\bbbr^k$ defined by
$$
g(x)=(d(f(x),y_1),\ldots, d(f(x),y_k))
$$
will be called the {\em projection of $f$ associated with the points $y_1,\ldots,y_k$.}

The mapping $\pi:X\to\bbbr^k$, $\pi(y)=(d(y,y_1),\ldots,d(y,y_k))$ is Lipschitz. Since
$g=\pi\circ f$, we conclude that if $f$ is Lipschitz, then its projection
$g=\pi\circ f$ is Lipschitz too.

A measurable function $g:E\to\bbbr$ defined in a measurable set $E\subset\bbbr^k$
is said to be approximately
differentiable at $x\in E$ if there is a measurable set $E_x\subset E$
and a linear function $L:\bbbr^n\to\bbbr$ such that $x$ is a density point of $E_x$ and
$$
\lim_{E_x\ni y\to x} \frac{g(y)-g(x)-L(y-x)}{|y-x|} = 0.
$$
This definition is equivalent with other definitions that one can find in the literature.
The approximate derivative $L$ is unique (if it exists) and it is denoted by $\ap Dg(x)$.
Lipschitz functions $g:E\to\bbbr$ are approximately differentiable a.e. (by the McShane extension and the Rademacher theorem).
In the case of mappings into $\bbbr^k$
approximate differentiability means approximate differentiability of each component.

\begin{theorem}
\label{main}
Let $X$ be a metric space, let $E\subset\bbbr^k$ be measurable, and let $f:E\to X$ be a Lipschitz mapping. Then the following statements are equivalent:
\begin{enumerate}
\item $\H^k(f(E)) = 0$;
\item For any Lipschitz mapping $\vi:X\to \bbbr^k$, we have $\H^k(\vi(f(E))) = 0$;
\item For any collection of distinct points $\{y_1, y_2, \dots, y_k\}\subset X$, the associated projection 
$g:E\to \bbbr^k$ of $f$ satisfies $\H^k(g(E)) = 0$;
\item For any collection of distinct points $\{y_1, y_2, \dots, y_k\}\subset X$, the associated projection 
$g:E\to\bbbr^k $ of $f$ satisfies $ \rank(\ap Dg(x)) < k$ for $\H^k$-a.e. $x\in E$.
\end{enumerate}
\end{theorem}
Here $\H^k$ stands for the $k$-dimensional Hausdorff measure.
\begin{remark}
It follows from the proof that in conditions (3) and (4) we do not have to consider all families $\{y_1, y_2, \dots, y_k\}\subset X$ of distinct points,
but it suffices to consider such families with points $y_i$ taken from a given countable and dense subset of $f(E)$.
\end{remark}

The implications from (1) to (2) and from (2) to (3) are obvious. 
The equivalence between (3) and (4) easily follows from the classical change of variables formula which states that if
$g:\bbbr^k\supset E\to\bbbr^k$ is Lipschitz, then
\begin{equation}
\label{change}
\int_E |J_g(x)|\, d\H^k(x) = \int_{g(E)} N_g(y,E)\, d\H^k(y).
\end{equation}
Here $J_g$ stands for the Jacobian of $g$ and $N_g(y,E)$ is the number of points in the preimage
$g^{-1}(y)\cap E$, see e.g. \cite{EG,federer,hajlasz2}.
Therefore, it remains to prove the implication (4) to (1) which is the most difficult 
part of the theorem. We will deduce it from another result which deals with Lipschitz mappings into $\ell^\infty$, see Theorem~\ref{intulinfty}.

Note that in general it may happen for a subset $A\subset X$ that
$\H^k(A)>0$, but for all Lipschitz mappings $\vi:X\to\bbbr^k$, $\H^k(\vi(A))=0$. For example
the Heisenberg group $\Heis^n$ satisfies $\H^{2n+2}(\Heis^n)=\infty$, but $\H^{2n+2}(\vi(\Heis^n))=0$
for all Lipschitz mappings $\vi:\Heis^n\to\bbbr^{2n+2}$, see \cite[Section~11.5]{DS}. Hence the implication from
(2) to (1) has to use in an essential way the assumption that $A=f(E)$ is a Lipschitz image of a Euclidean set.
Since by \cite[Section~11.5]{DS} the condition (2) is satisfied for $\Heis^n$ with $k=2n+2$, we conclude that $\Heis^n$ is purely $(2n+2)$-unrectifiable. 
For more general results see Theorem~\ref{pure k unrect} in Section~\ref{heisenberg} and Theorem~\ref{main-v} in Section~\ref{ap}.

Theorem~\ref{main} is related to the work of Kirchheim \cite{kirchheim} and Ambrosio-Kirchheim \cite{ambrosiok} on metric differentiability and the 
general area formula for mappings into arbitrary metric spaces. However, our approach in this paper is elementary and does not involve 
neither the Kirchheim-Rademacher theorem \cite[Theorem 2]{kirchheim} nor any kind of the 
area formula for mappings into arbitrary metric spaces \cite[Theorem 5.1]{ambrosiok}.

Although conditions (3) and (4) are necessary and sufficient for the validity of (1), often it is not easy to verify them.
The problem is that even if $X$ is smooth, the distance  function $y\mapsto d(y,y_i)$ is not smooth at $y_i$
and we need to consider such distance fucntions for $y_i$ from a dense subset of $X$,
thus creating singularities everywhere in $X$. Actually a collection of such distance functions gives an isometric embedding of $X$ into
$\ell^\infty$ (for a more precise statement see Theorem~\ref{intulinfty} and the proof of Theorem~\ref{main} which shows how Theorem~\ref{main} follows from
Theorem~\ref{intulinfty}). In applications we often deal with spaces $X$ that have some sort of smoothness 
(like Heisenberg groups or more general Carnot-Carath\'eodory spaces) and often for such spaces there is a more natural 
Lipschitz mapping $\Phi:X\to\bbbr^N$, than the embedding into $\ell^\infty$, a mapping that takes into account the structure of $X$.
In Section~\ref{carnot} we state a suitable version of Theorem~\ref{main} (Theorem~\ref{main-quasi}) and 
in Section~\ref{ap} we show how it applies to Carnot-Carath\'eodory spaces.

The paper is organized as follows. In Section~\ref{lipschitz-infty} we prove a version of the Sard theorem for Lipschitz mappings into $\ell^\infty$.
We also prove Theorem~\ref{main} as a simple consequence of this result.
In Section~\ref{heisenberg} we provide a new proof of the unrectifiability of the Heisenberg group as
a consequence of Theorem~\ref{main}. In the proof we will encounter a problem with the lack of smoothness of the distance of the function
$y\mapsto d(y,y_i)$.
In Section~\ref{carnot} we will generalize Theorem~\ref{main} in a way that it will easily apply to general Carnot-Carath\'eodory spaces
(including Heisenberg groups). This approach will allow us to avoid singularities of the distance function. Applications will be presented in
Section~\ref{ap}.

Our notation is fairly standard. By $C$ we will denote various positive constant whose value may change in a single string of estimates.
By writing $C=C(k)$ we mean that the constant $C$ depends on $k$ only. $\H^s$ will denote the $s$-dimensional Hausdorff measure.
We will also write $\H^k$ to denote the Lebesgue measure on $\bbbr^k$.
Sometimes in order to emphasize that the Hausdorff measure is defined with respect to a metric $d$ we will write $\H^s_d$.
If $V$ is a Banach space, then $\H^s_V$ denotes the Hausdorff measure with respect to the norm metric of $V$.
By $\H^s_\infty$ we will denote the Hausdorff content which is defined as the infimum of $\sum_{i=1}^\infty r_i^s$
over all coverings by balls of radii $r_i$. Clearly $\H^s_\infty$ is an outer measure and
$\H^s(A)=0$ if and only if $\H^s_\infty(A)=0$. The barred integral will denote the integral average
$\barcal_E f\, d\mu=\mu(E)^{-1}\int_E f\, d\mu$.

\noindent
{\sc Acknowledgements.} We would like to thank J. Pinkman for introducing us to the work of Heisenberg.

\section{Lipschitz mappings into $\ell^\infty$}
\label{lipschitz-infty}

A measurable function coincides with a continuous function outside a set of an arbitrarily small measure.
This is the Lusin property of measurable functions. The following result due to Federer shows a similar $C^1$-Lusin property of 
a.e. differentiable functions, \cite{whitney}.
\begin{lemma}[Federer]
\label{federer}
If $f:\Omega\to\bbbr$ is differentiable a.e.
on an open set $\Omega\subset\bbbr^k$, then for any
$\eps>0$ there is a function $g\in C^1(\bbbr^k)$ such that
$$
\H^k\left(\left\{x\in \Omega:\, f(x)\neq g(x)\right\}\right)<\eps.
$$
\end{lemma}
The original proof was based on the Whitney extension theorem; for another, more direct, approach,
see \cite[Theorem~1.69]{MalyZ}.

In particular if $E\subset\bbbr^k$ is measurable and $f:E\to\bbbr$ is Lipschitz, then
$f$ can be extended to a Lipschitz function $\tilde{f}:\bbbr^k\to\bbbr$ (McShane) to which
the above theorem applies. Hence for any $\eps>0$ there is $g\in C^1(\bbbr^k)$ such that
$$
\H^k\left(\left\{x\in E:\, f(x)\neq g(x)\right\}\right)<\eps.
$$
Note that at almost all points of the set where $f=g$ we have that $\ap Df(x)=Dg(x)$.
This holds true at all density points of the set $\{f=g\}$.

Let now $f=(f_1,f_2,\ldots):\bbbr^k\supset E\to\ell^\infty$ be an $L$-Lipschitz mapping.
Then the components $f_i:E\to\bbbr$ are also $L$-Lipschitz. Hence for $\H^k$-almost all points
$x\in E$, all functions $f_i$, $i\in\bbbn$ are approximately differentiable at $x\in E$.
We define the approximate derivative of $f$ componentwise
$$
\ap Df(x)=(\ap Df_1(x),\ap Df_2(x),\ldots).
$$
For each $i\in\bbbn$, $\ap Df_i(x)$ is a vector in $\bbbr^k$ with component bounded by $L$.
Hence $\ap Df(x)$ can be regarded as an $k\times\infty$ matrix of real numbers bounded by $L$, i.e.
$$
\ap Df(x)\in(\ell^\infty)^k,
\qquad
\Vert\ap Df\Vert_\infty\leq L,
$$
where the norm in $(\ell^\infty)^k$ is defined as the supremum over all entries in the
$k\times\infty$ matrix. The meaning of the rank of the $k\times\infty$ matrix $\ap Df(x)$ is clear;
it is the dimension of the linear subspace of $\bbbr^k$ spanned by the vectors $\ap Df_i(x)$, $i\in\bbbn$. Hence
$\rank(\ap Df(x))\leq k$ a.e.

If $f:\Omega\to\ell^\infty$ is Lipschitz, where $\Omega\subset\bbbr^k$ is open,
components of $f$ are differentiable a.e. and we will write $Df(x)$ in place of $\ap Df(x)$.

The next theorem is the main result of this section. It is a crucial step in the remaining implication (4) to (1)
of Theorem~\ref{main}.
The proof of Theorem~\ref{intulinfty} is based on ideas similar to those developed in \cite[Section~7]{BHW}.
\begin{theorem}
\label{intulinfty}
Let $E\subset\bbbr^k$ be measurable and let $f:E\to\ell^\infty$ be a Lipschitz mapping. Then
$\H^k(f(E))=0$ if and only if
$\rank(\ap Df(x))<k$, $\H^k$-a.e. in $E$.
\end{theorem}
Before we prove this result we will show how to use it to complete the proof of Theorem~\ref{main}.

\begin{proof}[Proof of Theorem~\ref{main}]
As we already pointed out in the Introduction, it remains to prove the implication (4) to (1).
Although we do not assume that $X$ is separable,
the image $f(E)\subset X$ is separable and hence it can be isometrically embedded into $\ell^\infty$.
More precisely let $\{y_i\}_{i=1}^\infty\subset f(E)$ be a dense subset and let $y_0\in f(E)$.
Then it is well-known and easy to prove that the mapping
$$
f(E)\ni y\mapsto\kappa(y) = \{d(y,y_i)-d(y_i,y_0)\}_{i=1}^\infty\in\ell^\infty
$$
is an isometric embedding of $f(E)$ into $\ell^\infty$. It is the so called Kuratowski embedding.
Clearly
$$
\H^k_d(f(E))=\H^k_{\ell^\infty}((\kappa\circ f)(E)),
$$
where subscripts indicate metrics with respect to which we define the Hausdorff measures.
It remains to prove that
$\H^k_{\ell^\infty}((\kappa\circ f)(E))=0$.
Since
$$
(\kappa\circ f)(x)=\{d(f(x),y_i)-d(y_i,y_0)\}_{i=1}^\infty,
$$
it easily follows from the assumptions that 
$$
\rank (\ap D(\kappa\circ f))<k
\quad
\mbox{$\H^k$-a.e. in $E$.}
$$
Hence (1) follows from Theorem~\ref{intulinfty}.
\end{proof}

Thus it remains to prove Theorem~\ref{intulinfty}. Before doing this
let us make some comments explaining why it is not easy. Theorem~\ref{intulinfty}
is related to the Sard theorem for Lipschitz mappings which states that if
$f:\bbbr^k\to\bbbr^m$, $m\geq k$ is Lipschitz, then
$$
\H^k(f(\{x\in\bbbr^k:\rank Df(x)<k\}))=0.
$$
The standard proof of this fact \cite[Theorem~7.6]{Mattila} is based on the observation that if
$\rank Df(x)<k$, then for any $\eps>0$ there is $r>0$ such that
$$
|f(z)-f(x)-Df(x)(z-x)|<\eps r
\quad
\mbox{for $z\in B(x,r)$}
$$
and hence 
$$
\dist(f(z),W_x)\leq \eps r
\quad
\mbox{for $z\in B(x,r)$,}
$$
where $W_x=f(x)+Df(x)(\bbbr^k)$ is an affine subspace of $\bbbr^m$ of dimension 
less than or equal to $k-1$. 
That means $f(B(x,r))$ is contained in a thin neighborhood of an ellipsoid of dimension
no greater than $k-1$ and hence we can cover it by
$C(L/\eps)^{k-1}$ balls of radius $C\eps r$, where $L$ is the Lipschitz constant of $f$.
Now we use covering by these balls with the help of Vitali's lemma to estimate the 
Hausdorff content of the image of the critical set. For more details, see \cite[Theorem~7.6]{Mattila}.

The proof described above employs the fact that $f$ is Frechet differentiable
and hence this argument {\em cannot} be applied to the case of mappings into $\ell^\infty$,
because in general Lipschitz mappings into $\ell^\infty$ are not Frechet differentiable,
i.e. in general the image of $f(B(x,r)\cap E)$ is not well approximated by the tangent mapping
$\ap Df(x)$. To overcome this difficulty we need to investigate the structure of the set
$\{\ap Df(x)<k\}$ using arguments employed in the proof of the general case of the Sard theorem
for $C^n$ mappings, \cite{sternberg}. In particular we will need to use a version of the implicit function
theorem.

In the proof of Theorem~\ref{intulinfty} we will also need the following result which is of independent interest.
\begin{proposition}
\label{diameter}
Let $D\subset\bbbr^k$ be a bounded and convex set with non-empty interior and let $f:D\to\ell^\infty$ be 
an $L$-Lipschitz mapping. Then
$$
\diam (f(D))\leq C(k)L\frac{(\diam D)^k}{\H^k(D)}\, \H^k(D\setminus A)^{1/k}
$$
where
$$
A=\{x\in D:\, Df(x)=0\}.
$$
In particular if $D$ is a cube or a ball, then
\begin{equation}
\label{image-cube}
\diam (f(D))\leq C(k)L\H^k(D\setminus A)^{1/k}
\end{equation}
\end{proposition}
\begin{proof}
We will need two well-known facts.
\begin{lemma}
\label{2.4}
If $E\subset\bbbr^k$ is measurable, then
$$
\int_E\frac{dy}{|x-y|^{k-1}}\leq C(k)\H^k(E)^{1/k}.
$$
\end{lemma}
\begin{proof}
Let $B=B(x,r)\subset\bbbr^k$ be a ball such that $\H^k(B)=\H^k(E)$. Then
$$
\int_E\frac{dy}{|x-y|^{k-1}}\leq\int_B\frac{dy}{|x-y|^{k-1}} = C(k)r=C'(k)\H^k(E)^{1/k}.
$$
\end{proof}
For the next lemma see for example \cite[Lemma~7.16]{EG}.
\begin{lemma}
If $D\subset\bbbr^k$ is a bounded and convex set with non-empty interior and if
$u:D\to\bbbr$ is Lipschitz continuous, then
$$
|u(x)-u_D|\leq \frac{(\diam D)^k}{k\H^k(D)}\int_D\frac{|\nabla u(y)|}{|x-y|^{k-1}}\, dy
\quad
\mbox{for all $x\in D$,}
$$
where
$u_D=\barcal_D u(x)\, dx$.
\end{lemma}
Now we can complete the proof of Proposition~\ref{diameter}.
If $Df(x)=0$, then $\nabla f_i(x)=0$ for all $i\in\bbbn$.
For each $i\in\bbbn$ we have
\begin{eqnarray*}
|f_i(x)-{f_i}_D|
& \leq &
\frac{(\diam D)^k}{k\H^k(D)}\int_D \frac{|\nabla f_i(y)|}{|x-y|^{k-1}}\, dy 
\leq
\frac{L(\diam D)^k}{k\H^k(D)}\int_{D\setminus A}\frac{dy}{|x-y|^{k-1}} \\
& \leq &
C(k)L\frac{(\diam D)^k}{\H^k(D)}\, \H^k(D\setminus A)^{1/k}.
\end{eqnarray*}
Hence for all $x,y\in D$
$$
|f_i(x)-f_i(y)|\leq |f_i(x)-{f_i}_D| + |f_i(y)-{f_i}_D|
\leq 2C(k)L\frac{(\diam D)^k}{\H^k(D)}\, \H^k(D\setminus A)^{1/k}.
$$
Taking supremum over $i\in\bbbn$ yields
$$
\Vert f(x)-f(y)\Vert_\infty \leq 2C(k)L\frac{(\diam D)^k}{\H^k(D)}\, \H^k(D\setminus A)^{1/k}
$$
and the result follows upon taking supremum over all $x,y\in D$.
\end{proof}

\begin{proof}[Proof of Theorem~\ref{intulinfty}]
The implication from left to right is easy. Suppose that $\H^k(f(E))=0$.
For any positive integers $i_1<i_2<\ldots<i_k$ the projection
$$
\ell^\infty\ni (y_1,y_2,\ldots)\to (y_{i_1},y_{i_2},\ldots,y_{i_k})\in\bbbr^k
$$
is Lipschitz continuous and hence the set
$$
(f_{i_1},\ldots,f_{i_k})(E)\subset\bbbr^k
$$
has $\H^k$-measure zero. It follows from the change of variables formula \eqref{change}
that the matrix $[\partial f_{i_j}/\partial x_\ell]_{j,\ell=1}^k$ of approximate partial
derivatives has rank less than $k$ almost everywhere in $E$. Since this is true for any choice
of $i_1<i_2<\ldots<i_k$, we conclude that $\rank(\ap Df(x))<k$ a.e. in $E$.

Suppose now that $\rank (\ap Df(x))<k$ a.e. in $E$. We need to prove that $\H^k(f(E))=0$.
This implication is more difficult. 
Since $f_i:E\to\bbbr$ is Lipschitz continuous, for any $\eps>0$ there is $g_i\in C^1(\bbbr^n)$ such that
$$
\H^k(\{x\in E:\, f_i(x)\neq g_i(x)\})<\eps/2^i.
$$
Moreover $\ap Df_i(x)=Dg_i(x)$ for almost all points of the set where $f_i=g_i$.
Hence there is a measurable set $F\subset E$ such that $\H^k(E\setminus F)<\eps$ and
$$
f=g
\quad \mbox{and} \quad 
\ap Df(x)=Dg(x)
\quad
\mbox{in $F$}
$$
where
$$
g=(g_1,g_2,\ldots),
\quad
Dg=(Dg_1,Dg_2,\ldots).
$$
It suffices to prove that
$\H^k(f(F))=0$, because we can exhaust $E$ with sets $F$ up to a subset of measure zero and $f$
maps sets of measure zero to sets of measure zero. Let
$$
\tilde{F}=\{x\in F:\, \rank(\ap Df(x))=\rank Dg(x)<k\}.
$$
Since $\H^k(F\setminus\tilde{F})=0$, it suffices to prove that $\H^k(f(\tilde{F}))=0$.
For $0\leq j\leq k-1$, let
$$
K_j=\{x\in\tilde{F}:\rank Dg(x)=j\}.
$$
Since $\tilde{F}=\bigcup_{j=0}^{k-1} K_j$, it suffices to prove that $\H^k(f(K_j))=0$
for any $0\leq j\leq k-1$. Again, by removing a subset of measure zero we can assume that all points 
of $K_j$ are density points of $K_j$. To prove that $\H^k(f(K_j))=0$ we
need to make a change of variables in $\bbbr^k$, but only when $j\geq 1$.

If $x\in \bbbr^k\setminus F$, the sequence $(g_1(x),g_2(x),\ldots)$ is not necessarily bounded.
Let $V$ be the linear space of all real sequences $(y_1,y_2,\ldots)$. Clearly $g:\bbbr^k\to V$.
We do not equip $V$ with any metric structure. Note that
$g|_F:F\to\ell^\infty\subset V$, because $g$ coincides with $f$ on $F$.
\begin{lemma}
\label{changev}
Let $1 \leq j \leq k - 1$ and $x_0 \in K_j$. Then there exists a neighborhood 
$x_0 \in U \subset \bbbr^k $, a diffeomorphism $\Phi: U \subset \bbbr^k \to \Phi(U) \subset \bbbr^k$, and a composition of a translation 
(by a vector from $ \ell^\infty $) with a permutation of variables $\Psi \colon V \to V$ such that
\begin{itemize}
\item $\Phi^{-1}(0) = x_0$ and $\Psi(g(x_0)) = 0$;
\item There is $\eps > 0$ such that for $x = (x_1, x_2, \dots, x_k) \in B(0,\eps) \subset \bbbr^k$ and $i = 1, 2, \dots, j$,
$$
\left( \Psi \circ g \circ \Phi^{-1} \right)_i(x) = x_i,
$$
i.e., $\Psi \circ g \circ \Phi^{-1}$ fixes the first $j$ variables in a neighborhood of $0$.
\end{itemize}
\end{lemma}
\begin{proof}
By precomposing $g$ with a translation of $\bbbr^k$ by the vector $x_0$ and postcomposing it with a translation 
of $V$ by the vector $-g(x_0) = -f(x_0) \in \ell^\infty$ we may assume that $x_0 = 0$ and $g(x_0) = 0$.
A certain $j \times j$ minor of $Dg(x_0)$ has rank $j$. By precomposing $g$ with a permutation of $j$ variables in 
$\bbbr^k$ and postcomposing it with a permutation of $j$ variables in $V$ we may assume that
\begin{equation}
\label{rankofg}
\rank\left[ \frac{\partial g_m}{\partial x_\ell}(x_0) \right]_{1 \leq m , \ell \leq j} = j.
\end{equation}
Let $H:\bbbr^k \to \bbbr^k$ be defined by
$$
H(x) = (g_1(x), \dots, g_j(x), x_{j+1}, \dots, x_k).
$$
It follows from \eqref{rankofg} that $J_H(x_0) \neq 0$ and hence $H$ is a diffeomorphism in a 
neighborhood of $x_0=0 \in \mathbb{R}^k$.
It suffices to observe that for all $i = 1, 2, \dots, j$,
$$
\left( g \circ H^{-1} \right)_i(x) = x_i.
$$
\end{proof}
In what follows, by cubes, we will mean cubes with edges parallel to the coordinate axes in $\bbbr^k$. 
It suffices to prove that any point $x_0\in K_j$ has a cubic neighborhood whose intersection with $K_j$ is mapped onto a set of $\H^k$-measure zero. 
Since we can take cubic neighborhoods to be arbitrarily small, the change of variables from Lemma~\ref{changev} allows us to assume that 
\begin{equation}
\label{star}
K_j\subset (0,1)^k,
\quad
g_i(x)=x_i
\quad
\mbox{for $i=1,2,\ldots,j$ and $x\in [0,1]^k$.}
\end{equation}
Indeed, according to Lemma~\ref{changev} we can assume that $x_0=0$ and that $g$ fixes the first $j$ variables in a neighborhood of $0$.
The neighborhood can be very small, but a rescaling argument allows us to assume that it contains a unit cube $Q$ around $0$. 
Translating the cube we can assume that $Q=[0,1]^k$.
If $x \in K_j$, since $\rank Dg(x) = j$ and $g$ 
fixes the first $j$ coordinates, the derivative of $g$ in directions orthogonal to the first $j$ coordinates equals zero at $x$,
$\partial g_\ell(x)/\partial x_i=0$ for $i=j+1,\ldots,k$ and any $\ell$.
\begin{lemma}
\label{mainlemma}
Under the assumptions \eqref{star}
there exists a constant $C = C(k) > 0$ such that for any integer $m \geq 1$, and every $x \in K_j$, 
there is a closed cube $Q_x\subset [0,1]^k$ with edge length $d_x$
centered at $x$ with the property that 
$f(K_j \cap Q_x) = g(K_j \cap Q_x)$ can be covered by $m^j$ balls in $\ell^\infty$, each of radius $C L d_x m^{-1}$,
where $L$ is the Lipschitz constant of $f$.
\end{lemma} 
The theorem is an easy consequence of this lemma through a standard application of the $5r$-covering lemma, \cite[Theorem~1.2]{heinonen}. 
First of all observe that cubes 
with sides parallel to coordinate axes in $\bbbr^k$ are balls with respect to
the $\ell^\infty_k$ metric
$$
\Vert x-y\Vert_\infty = \max_{1\leq i\leq k}|x_i-y_i|.
$$
Hence the $5r$-covering lemma applies to families of cubes in $\bbbr^k$.
By $5^{-1}Q$ we will denote a cube concentric with $Q$ and with $5^{-1}$ times the diameter.
The cubes $\{ 5^{-1}Q_x\}_{x\in K_j}$ form a covering of $K_j$. Hence we can select disjoint cubes
$\{ 5^{-1}Q_{x_i}\}_{i=1}^\infty$ such that
$$
K_j\subset\bigcup_{i=1}^\infty Q_{x_i}.
$$
If $d_i$ is the edge length of $Q_{x_i}$, then $\sum_{i=1}^\infty (5^{-1}d_i)^k\leq 1$,
because the cubes $5^{-1}Q_{x_i}$ are disjoint and contained in $[0,1]^k$. Hence
$$
\H^k_\infty(f(K_j))\leq
\sum_{i=1}^\infty \H^k_\infty(f(K_j\cap Q_{x_i}))\leq
\sum_{i=1}^\infty m^j(CLd_im^{-1})^k \leq
5^k C^k L^k m^{j-k}.
$$
Since the exponent $j-k$ is negative, and $m$ can be arbitrarily large we conclude that
$\H^k_\infty(f(K_j))=0$ and hence $\H^k(f(K_j))=0$.
Thus it remains to prove Lemma~\ref{mainlemma}.

\begin{proof}[Proof of Lemma~\ref{mainlemma}]
Various constants $C$ in the proof below will depend on $k$ only. Fix an integer $m\geq 1$.
Let $x\in K_j$. Since every point in $K_j$ is a density point of $K_j$, there is a closed cube $Q\subset [0,1]^k$ 
centered at $x$ of edge length 
$d$ such that
\begin{equation}
\label{jeden}
\H^k(Q\setminus K_j)<m^{-k}\H^k(Q)=m^{-k}d^k.
\end{equation}
By translating the coordinate system in $\bbbr^k$ we may assume that
$$
Q=[0,d]^j\times[0,d]^{k-j}.
$$
Each component of $f:Q\cap K_j\to\ell^\infty$ is an $L$-Lipschitz function. Extending each component to
an $L$-Lipschitz function on $Q$ results in an $L$-Lipschitz extension $\tilde{f}:Q\to \ell^\infty$.
This is well-known and easy to check.

Divide $[0,d]^j$ into $m^j$  cubes with pairwise disjoint interiors, each of edge length $m^{-1}d$.
Denote the resulting cubes  by $Q_\nu$, $\nu\in\{1,2,\ldots,m^j\}$. It remains to prove that
$$
f((Q_\nu\times [0,d]^{k-j})\cap K_j)\subset\tilde{f}(Q_\nu\times [0,d]^{k-j})
$$
is contained in a ball (in $\ell^\infty$) of radius $CLdm^{-1}$.
It follows from \eqref{jeden} that
$$
\H^k((Q_\nu\times [0,d]^{k-j})\setminus K_j) \leq \H^k(Q\setminus K_j)<m^{-k}d^k.
$$
Hence
$$
\H^k((Q_\nu\times[0,d]^{k-j})\cap K_j)>(m^{-j}-m^{-k})d^k.
$$
This estimate and the Fubini theorem imply that there is $\rho\in Q_\nu$ such that
$$
\H^{k-j}((\{\rho\}\times[0,d]^{k-j})\cap K_j)>(1-m^{j-k})d^{k-j}.
$$
Hence
$$
\H^{k-j}((\{\rho\}\times[0,d]^{k-j})\setminus K_j)<m^{j-k}d^{k-j}.
$$
It follows from \eqref{image-cube} with $k$ replaced by $k-j$ that
\begin{equation}
\label{dwa}
\diam_{\ell^\infty}( \tilde{f}(\{\rho\}\times[0,d]^{k-j}))\leq
CL\H^{k-j}((\{\rho\}\times[0,d]^{k-j})\setminus K_j)^{1/(k-j)}\leq CLm^{-1}d.
\end{equation}
Indeed, the rank of the derivative of $g$ restricted to the slice 
$\{\rho\}\times[0,d]^{k-j}$ equals zero 
at the points of $(\{\rho\}\times [0,d]^{k-j})\cap K_j$
and this derivative coincides a.e. with the 
approximate derivative of $\tilde{f}$ restricted to $\{\rho\}\times[0,d]^{k-j}\cap K_j$
which by the property of $g$ must be zero as well.

Since the distance of any point in $Q_\nu\times[0,d]^{k-j}$ to
$\{\rho\}\times[0,d]^{k-j}$ is bounded by $Cm^{-1}d$ and $\tilde{f}$ is $L$-Lipschitz, \eqref{dwa}
implies that $\tilde{f}(Q_\nu\times [0,d]^{k-j})$ is contained in a ball of radius $CLdm^{-1}$,
perhaps with a constant $C$ bigger than that in \eqref{dwa}.
The proof of the lemma is complete.
\end{proof}
This also completes the proof of Theorem~\ref{intulinfty}.
\end{proof}

\section{Heisenberg groups}
\label{heisenberg}

As an application we will show one more proof of the well-known result
of Ambrisio-Kircheim \cite{ambrosiok} and Magnani \cite{magnani}
that the Heisenberg group $\Heis^n$ is purely $k$-unrectifiable for $k>n$.
Another proof was given in \cite{BHW} and our argument is related to the one given in 
\cite{BHW} in a sense that the proof of Theorem~\ref{intulinfty} is based on similar ideas.
We will not recall the definition of the Heisenberg group
as this is not the main subject of the paper. The reader may find a detailed
introduction for example in \cite{BHW}; we will follow notation used in that paper.
The following result is well-known, see for example Theorem~1.2 in \cite{BHW}.
\begin{lemma}
\label{low-rank} 
Let $k>n$ and let $E \subset \bbbr^{k}$ be a measurable set. 
If $f \colon E \to \Heis^n$ is locally Lipschitz continuous, then for $\H^k$-almost every point $x \in E$,
$\rank (\ap Df(x)) \leq n$.
\end{lemma}
The Heisenberg group $\Heis^n$ is homeomorphic to $\bbbr^{2n+1}$ and the identity mapping
${\rm id}:\Heis^n\to\bbbr^{2n+1}$ is locally Lipschitz continuous. Hence $f$ is locally Lipschitz as a mapping
into $\bbbr^{2n+1}$. The approximate derivative $\ap Df(x)$ is understood as the derivative of the mapping
into $\bbbr^{2n+1}$. As an application of Theorem~\ref{main} we will prove unrectifiability of $\Heis^n$.
\begin{theorem}
\label{pure k unrect}  
Let $k>n$ be positive integers. Let $E \subset \bbbr^k$ be a measurable set, and let $f \colon E \to \Heis^n$ 
be a Lipschitz mapping. Then $\H^k(f(E))=0$.
\end{theorem}
Here the Hausdorff measure in $\Heis^n$ is with respect to the Carnot-Carath\'eodory metric or with respect to 
the Kor\'anyi metric $d_K$ which is bi-Lipschitz equivalent to the Carnot-Carath\'eodory one.
\begin{proof}
Let $f:\bbbr^k\supset E\to\Heis^n$, $k>n$ be Lipschitz. We need to prove that $\H^k(f(E))=0$.
Recall that by Lemma~\ref{low-rank}, $\rank (\ap Df(x)) \leq n$.
Fix a collection of $k$ distinct points $y_i,\ldots,y_k$ in $\Heis^n$ and define the mapping
$g:\bbbr^k\supset E\to\bbbr^k$ as the projection of $f$
$$
g(x)=(d_K(f(x),y_1),\ldots,d_K(f(x),y_k)).
$$
The mapping $\pi:\Heis^n\to\bbbr^k$ defined by
$\pi(z)=(d_K(z,y_1),\ldots,d_K(z,y_k))$ is Lipschitz continuous, but it is not Lipschitz as a mapping
$\pi:\bbbr^{2n+1}\to\bbbr^k$. Hence it is not obvious that we can apply the chain rule to $g=\pi\circ f$ and conclude that
$\rank(\ap Dg(x))\leq n<k$ a.e. in $E$ which would imply $\H^k(f(E))=0$ by Theorem~\ref{main}.
To overcome this difficulty we use the fact that 
the Kor\'anyi metric $\bbbr^{2n+1}\ni z\mapsto d_K(z,y)\in\bbbr$ is $C^\infty$ on $\bbbr^{2n+1}\setminus \{ y\}$.
Hence the chain rule applies to $g=\pi\circ f$ on the set
$E\setminus(\bigcup_{i=1}^k E_i)$, where
$$
E_i=\{x\in E:\, f(x)=y_i\}
$$
and $\rank (\ap Dg(x))\leq n<k$ a.e. in $E\setminus(\bigcup_{i=1}^k E_i)$. If $x\in E_i$, then 
$f(x)\neq y_j$ for $j\neq i$ and 
$$
g(x)=(d_K(f(x),y_1),\ldots,d_K(f(x),y_{i-1}),0,d_K(f(x),y_{i+1}),\ldots,d_K(f(x),y_k)),
\quad
\mbox{for $x\in E_i$}.
$$
Thus $g=\pi_i\circ f$ on $E_i$, where
$$
\pi_i(z)=(d_K(z,y_1),\ldots,d_K(z,y_{i-1}),0,d_K(z,y_{i+1}),\ldots,d_K(z,y_k)).
$$
The function $\pi_i$ is smooth in a neighborhood of $y_i=f(x)$, $x\in E_i$ and hence the chain rule shows that
the approximate derivative of $g|_{E_i}$ has rank less than or equal $n<k$ a.e. in $E_i$. It remains to observe that
at almost all points of $E_i$ the approximate derivative of $g$ equals to that of $g|_{E_i}$.
\end{proof}

\section{Generalization of Theorem~\ref{main}}
\label{carnot}

\begin{definition}
We say that a metric space $(X,d)$ is {\em quasiconvex} if there is a constant $M\geq 1$ such that any two points
$x,y\in X$  can be connected by a curve $\gamma$ of length $\ell(\gamma)\leq Md(x,y)$.
\end{definition}
The next result is a variant of Theorem~\ref{main}.
\begin{theorem}
\label{main-quasi}
Suppose that $(X,d)$ is a complete and quasiconvex metric space and that $\Phi:X\to\bbbr^N$ is a Lipschitz
map with the property that for some constant $C_\Phi>0$ and all rectifiable curves $\gamma$ in $X$ we have
\begin{equation}
\label{BLD}
\ell(\gamma)\leq C_\Phi\ell(\Phi\circ\gamma).
\end{equation}
Then for any $k\geq 1$ and any Lipschitz map $f:\bbbr^k\supset E\to X$ defined on a measurable set $E\subset \bbbr^k$
the following conditions are equivalent.
\begin{enumerate}
\item $\H^k(f(E))=0$ in $X$;
\item $\H^k(\Phi(f(E)))=0$ in $\bbbr^N$;
\item $\rank (\ap D(\Phi\circ f))<k$, $\H^k$-a.e. in $E$.
\end{enumerate}
\end{theorem}
Since the set $f(E)$ is separable, $\H^k(f(E))=0$ if and only if every 
point in the set $f(E)$ has a neighborhood whose intersection with $f(E)$ has measure zero. 
This also implies that a local version of Theorem~\ref{main-quasi} is true: We can assume that
the space is quasiconvex in a neighborhood of each point, that $\Phi$ is locally Lpschitz 
continuous and that for each $x\in X$ there is a neighborhood $x\in U\subset X$ and a constant $C_{\Phi,U}$
such that \eqref{BLD} holds for all rectifiable curves $\gamma$ in $U$ with the constant $C_{\Phi,U}$.
The reader will have no problem to state a suitable version of the theorem. 

In the proof of Theorem~\ref{main} we embedded $f(E)$ isometrically into $\ell^\infty$ and we concluded the result from 
Theorem~\ref{intulinfty}. Here instead of the isometric embedding into $\ell^\infty$ we have the mapping $\Phi$.
The proof of Theorem~\ref{main-quasi} is similar to that of Theorem~\ref{intulinfty}
and for that reason our arguments will be sketchy, but an essential difficulty arises in the proof of the
counterpart of the estimate \eqref{dwa}. One of the reasons for this difficulty is that unlike $\ell^\infty$,
the space $X$ does not necessarily have the Lipschitz extension property and we cannot extend $f$ from 
$Q\cap K_j$ to a Lipschitz mapping $\tilde{f}:Q\to X$; we will need a slightly different argument and 
this part of the proof will be furnished with all the necessary details.

\begin{proof}[Proof of Theorem~\ref{main-quasi}]
The implication from (1) to (2) is obvious. If $N<k$, the equivalence between (2) and (3) is also obvious, so we
can assume that $N\geq k$. In that case the equivalence between (2) and (3) follows from the area formula
which generalizes \eqref{change} to the case when the target space may have larger dimension than the domain: 
If $h:\bbbr^k\supset E\to\bbbr^N$ is Lipschitz, then
$$
\int_E |J_h(x)|\, d\H^k(x)=\int_{h(E)} N_h(y,E)\, d\H^k(y),
$$
\cite{EG,federer}, and the observation that $|J_h(x)|=0$ if and only if $\rank(\ap Dh(x))<k$.
It remains to prove that (3) implies (1). Suppose that
$\rank(\ap D(\Phi\circ f))<k$ a.e. in $E$. For any $\eps>0$ there is a set $F\subset E$ and a mapping
$g=(g_1,\ldots,g_N)\in C^1(\bbbr^k,\bbbr^N)$ such that $\H^k(E\setminus F)<\eps$
and
$$
g=\Phi\circ f,
\quad
Dg=\ap D(\Phi\circ f),
\quad
\rank Dg<k
\quad
\mbox{on $F$.}
$$
Since $F=\bigcup_{j=0}^{k-1} K_j$, where
$$
K_j=\{x\in F:\rank Dg(x)=j\},
$$
it suffices to show that $\H^k(f(K_j))=0$. By removing a subset of measure
zero we can assume that all points of $K_j$ are the density points of $K_j$.
Since the problem is local in the nature using a variant of Lemma~\ref{changev} 
we can assume that 
\begin{equation}
\label{star2}
K_j\subset (0,1)^k,
\quad
g_i(x)=x_i
\quad
\mbox{for $i=1,2,\ldots,j$ and $x\in [0,1]^k$.}
\end{equation}
Now the result will follow from the following version of
Lemma~\ref{mainlemma}.
\begin{lemma}
\label{ml2}
Under the assumption \eqref{star2}
there is a constant $C=C(k)C_\Phi M\lip(\Phi)>0$
such that for any integer $m\geq 1$, and any  $x\in K_j$,
there is a closed cube $Q_x\subset [0,1]^k$ centered at $x$
of edge length $d_x$ such that $f(K_j\cap Q_x)$
can be covered by $m^j$ balls in $X$, each of radius $CLd_xm^{-1}$, where $L$ is the Lipschitz constant of $f$.
\end{lemma}
To prove the lemma we choose $Q\subset [0,1]^k$ with edge length $d$, centered at $x$
such that $\H^k(Q\setminus K_j)<m^{-k}d^k$.
We can assume that $Q=[0,d]^{k}$.
Divide $Q$ into $m^j$ rectangular boxes $Q_\nu\times[0,d]^{k-j}$. 
We need to show that $f((Q_\nu\times [0,d]^{k-j})\cap K_j)$ is contained in a ball of radius
$CLdm^{-1}$. We find $\rho\in Q_\nu$ such that
\begin{equation}
\label{square}
\H^{k-j}((\{\rho\}\times [0,d]^{k-j})\setminus K_j)<m^{j-k}d^{k-j}.
\end{equation}
By the volume argument every point in $\{\rho\}\times [0,d]^{k-j}$
is at the distance no more than $C(k)m^{-1}d$ to the set
$(\{\rho\}\times [0,d]^{k-j})\cap K_j$. Hence every point in
$Q_\nu\times [0,d]^{k-j}$, and thus every point in 
$(Q_\nu\times [0,d]^{k-j})\cap K_j$, is at the distance less than or equal to 
$C(k)m^{-1}d$ from the set $(\{\rho\}\times [0,d]^{k-j})\cap K_j$. Since $f$ is $L$-Lipschitz
it suffices to show that 
\begin{equation}
\label{diamX}
\diam_X f((\{\rho\}\times [0,d]^{k-j})\cap K_j)<CLdm^{-1}.
\end{equation}
This is the estimate that plays the role of \eqref{dwa}, but the proof has to be different now.
\begin{lemma}
\label{si}
Let $E\subset Q$ be a measurable subset of a cube $Q\subset\bbbr^n$.
For $x,y\in Q$ let $I_x(y)$ be the length of the intersection of the
interval $\overline{xy}$ with $E$, i.e.
$I_x(y)=\H^1(\overline{xy}\cap E)$.
Then there is a constant $C=C(n)>0$ such that for any $x\in Q$
\begin{equation}
\label{blabla}
\H^n(\{y\in Q:\, I_x(y)\leq C\H^n(E)^{1/n}\})>\frac{\H^n(Q)}{2}\, .
\end{equation}
\end{lemma}
The lemma says that if the measure of $E$ is small, then more than 50\% of the intervals
$\overline{xy}$ intersect $E$ along a short subset.
\begin{proof}
It suffices to show that for some constant $C=C(n)$
$$
\barint_Q I_x(y)\, dy \leq C\H^n(E)^{1/n}.
$$
Then \eqref{blabla} will be true with $C$ replaced by $2C$.
For $z\in S^{n-1}$ let $\delta(z)=\sup\{t>0:\, x+tz\in Q\}$.
An integral over $Q$ can be represented in the spherical coordinates centered at $x$ as follows
\begin{equation}
\label{spherical}
\int_Q f(y)\, dy = \int_{S^{n-1}}\int_0^{\delta(z)}
f(x+tz)t^{n-1}\, dt\, d\sigma(z).
\end{equation}
If $z\in S^{n-1}$, then
$$
I_x(x+tz)\leq I_x(x+\delta(z)z)=\int_0^{\delta(z)}\chi_E(x+\tau z)\, d\tau.
$$
We have
\begin{eqnarray}
\barint_Q I_x(y)\, dy
& = & 
\frac{1}{\H^n(Q)} \int_{S^{n-1}}\int_0^{\delta(z)} t^{n-1} I_x(x+tz)\, dt\, d\sigma(z) \nonumber \\
& \leq &
\frac{1}{\H^n(Q)} \int_{S^{n-1}} \int_0^{\delta(z)} t^{n-1}
\int_0^{\delta(z)} \chi_E(x+\tau z)\, d\tau\, dt\, d\sigma(z) \nonumber \\
& \leq &
\frac{1}{\H^n(Q)} \int_{S^{n-1}} \int_0^{\diam Q} t^{n-1}\, dt \int_0^{\delta(z)} \chi_E(x+\tau z)\, d\tau\, d\sigma(z) \nonumber \\
& = & 
C(n)\int_{S^{n-1}}\int_0^{\delta(z)}
\frac{\chi_E(x+\tau z)}{\tau^{n-1}}\, \tau^{n-1}\, d\tau\, d\sigma(z) \label{2014}\\
& = & C\int_Q\frac{\chi_E(y)}{|x-y|^{n-1}}\, dy 
\leq
C\H^n(E)^{1/n} \nonumber
\end{eqnarray}
by Lemma~\ref{2.4}. Equality \eqref{2014} follows from \eqref{spherical}.
\end{proof}
Now under the assumptions of the lemma, if $x,y\in Q$, we can find $z\in Q$ such that
$I_x(z)+I_y(z)\leq C\H^n(E)^{1/n}$, i.e. the curve $\overline{xz}+\overline{zy}$ connecting $x$ to $y$
has length no bigger than $2\diam Q$ and it intersects the set $E$ along a subset of length
less than or equal to $C\H^n(E)^{1/n}$.
Applying it to $n=k-j$,
$Q=\{\rho\}\times [0,d]^{k-j}$, and 
$E=(\{\rho\}\times [0,d]^{k-j})\setminus K_j$, every pair of points $x,y\in Q\cap K_j$ we can be connected by
a curve $\gamma=\overline{xz}+\overline{zy}$ of length $\ell(\gamma)\leq 2d\sqrt{k-j}$ 
(two times the diameter of the cube) whose intersection with the
complement of $K_j$ has length no more than $C(k)m^{-1}d$ by \eqref{square}.
We can parametrize $\gamma$ by arc-length $\gamma:[0,\ell(\gamma)]\to \{\rho\}\times [0,d]^{k-j}$
as a $1$-Lipschitz curve. The mapping
$f\circ \gamma$ is $L$-Lipschitz and defined on a subset $\gamma^{-1}(K_j)$.
It uniquely extends to the closure of  $\gamma^{-1}(K_j)$
(because it is Lipschitz and $X$ is complete). The complement of this set
consists of countably many open intervals of total length bounded by $C(k)m^{-1}d$. Since the space $X$ is quasiconvex we can extend
$f\circ \gamma$ from the closure of $\gamma^{-1}(K_j)$ to
$\widetilde{f\circ\gamma}:[0,\ell(\gamma)]\to X$ as an $ML$-Lipschitz curve connecting $x$ to $y$; 
here $M$ is the quasiconvexity constant of the space $X$.
The curve
$$
\Phi\circ(\widetilde{f\circ\gamma}):[0,\ell(\gamma)]\to\bbbr^N
$$
is $\lip(\Phi)ML$-Lipschitz.
Note that on the set $\gamma^{-1}(K_j)$ this curve coincides with $g\circ\gamma$
and hence for a.e. $t\in\gamma^{-1}(K_j)$ we have
$$
(\Phi\circ(\widetilde{f\circ\gamma}))'(t)=(g\circ\gamma)'(t)=0.
$$
Hence the length of the curve $\Phi\circ(\widetilde{f\circ\gamma})$ is bounded by
\begin{eqnarray*}
\ell(\Phi\circ(\widetilde{f\circ\gamma})) 
& = &
\int_0^{\ell(\gamma)} |(\Phi\circ(\widetilde{f\circ\gamma}))'(t)|\, dt \leq
\lip(\Phi)ML \H^1([0,\ell(\gamma)]\setminus\gamma^{-1}(K_j)) \\
& \leq &
\lip(\Phi)MLC(k)m^{-1}d.
\end{eqnarray*} 
Now \eqref{BLD} implies that
$$
d(f(x),f(y))\leq 
\ell(\widetilde{f\circ\gamma}) \leq
C_\Phi\ell(\Phi\circ(\widetilde{f\circ\gamma}))\leq
C_\Phi \lip(\Phi)MLC(k)m^{-1}d.
$$
Since this is true for all $x,y\in \{\rho\}\times [0,d]^{k-j}\cap K_j$,
\eqref{diamX} follows. The proof is complete.
\end{proof}

\section{Applications}
\label{ap}

\subsection{Mappings of bounded length distortion}

\begin{definition}
A mapping $f:X\to Y$ between metric spaces is 
said to have the {\em weak bounded length distortion} property (weak BLD) if
there is a constant $C\geq 1$ such that for all rectifiable
curves $\gamma$ in $X$ we have
\begin{equation}
\label{w-BLD}
C^{-1}\ell_X(\gamma)\leq \ell_Y(f\circ\gamma)\leq C\ell_X(\gamma).
\end{equation}
\end{definition}

The class of mappings with bounded length distortion (BLD)
was introduced in \cite{martiov} under the assumption that
$f:\bbbr^n\supset\Omega\to\bbbr^n$ is a continuous mapping on 
an open domain such that it is open, discrete, sense preserving and
satisfies \eqref{w-BLD} for all curves $\gamma$ in $\Omega$.
A more general definition without any topological restrictions was given in 
\cite[Definition~2.10]{ledonne}. This definition is almost identical
to ours, but it was assumed that \eqref{w-BLD} was satisfied
for {\em all} curves $\gamma$ in $X$. The two notions are different:
it may happen that a mapping has the weak BLD property,
but some curves of infinite length in $X$ are mapped onto rectifiable curves
and hence such a mapping is not BLD in the sense of
\cite[Definition~2.10]{ledonne}. 
For example the identity mapping on the Heisenberg group
${\rm id}:\Heis^n\to\bbbr^{2n+1}$ satisfies the weak BLD condition locally.
However, any segment on the $t$-axis has infinite length in the metric
of $\Heis^n$ (actually its Hausdorff dimension equals $2$) and it is mapped
by the identity mapping to a segment in the $t$-axis in $\bbbr^{2n+1}$ of finite
Euclidean length.

As a consequence of Theorem~\ref{main-quasi} we obtain.
\begin{theorem}
\label{BLD5}
If a mapping $f:\bbbr^n\supset\Omega\to\bbbr^m$  
defined on an open set $\Omega\subset\bbbr^n$
has the weak BLD property, then $f$ is locally Lipschitz, $m\geq n$
and $\rank Df(x)=n$ a.e. in $\Omega$.
\end{theorem}
\begin{proof} For any $y\in B(x,r)\subset\Omega$, the segment $\overline{xy}$ is mapped on 
a curve of length bounded by $C|x-y|$. Hence $|f(x)-f(y)|\leq C|x-y|$.
Let $X$ be a closed ball contained in $\Omega$, equip it with the Euclidean metric and let $\Phi=f|_X:X\to\bbbr^m$. 
Let $E\subset X$ be the set of points where $\rank Df<n$ and
let $\iota:E\to X$ be the identity mapping.
According to Theorem~\ref{main-quasi}, 
$\H^n(E)=\H^n(\iota(E))=0$ if and only if $\rank (\ap D(\Phi\circ\iota))=\rank Df <n$, a.e. in $E$.
Since the last condition is satisfied by the definition of $E$, we conclude that
$\H^n(E)=0$, and hence $\rank Df(x)=n$ a.e. in $\Omega$, because $\Omega$ is a countable union of closed balls. This however, implies that $m\geq n$.
\end{proof}

Gromov proved in \cite[2.4.11]{gromov} that
any Riemannian manifold of dimension $n$ admits a
mapping into $\bbbr^n$ that preserves lengths of curves. It follows from Theorem~\ref{BLD5}
that the Jacobian of such mapping is different than zero a.e. and hence there is no such 
mapping into $\bbbr^m$ for $m<n$ (this result is known).

In \cite{martiov} it was proved that a mapping
$f:\bbbr^n\supset\Omega\to\bbbr^n$ is BLD (under the topological assumptions: open, discrete, sense preserving) 
if and only if $f$ is locally Lipschitz and $|J_f|\geq c>0$ a.e. We proved without any topological assumptions
that $|J_f|>0$ a.e.

\subsection{Carnot-Carath\'eodory spaces}

Let $X_1,X_2,\ldots,X_m$ be a family of vector fields defined on an open and connected set $\Omega\subset\bbbr^n$
with locally Lipschitz continuous coefficients. 
Assume that the vector fields are linearly independent at every point of $\Omega$ and that
for every compact set $K\subset\Omega$ 
$$
\inf_{p\in K}\inf_{i\in\{1,\ldots,m\}} |X_i(p)|>0.
$$
For $v=\sum_i a_i X_i(p)\in {\rm span}\, \{X_i(p),\ldots, X_m(p)\}$ we define
$$
|v|_H=\Big(\sum_{i=1}^m a_i^2\Big)^{1/2}.
$$
It follows from our assumptions that on compact subsets of $\Omega$, $|v|_H$
is comparable to the Euclidean length $|v|$ of the vector $v$, i.e. for
every compact set $K\subset\Omega$ there is a constant $C\geq 1$
such that
\begin{equation}
\label{2341}
C^{-1}|v|\leq |v|_H\leq C|v|
\quad
\mbox{for all $p\in K$ and all $v\in {\rm span}\, \{X_1(p),\ldots, X_m(p)\}$}.
\end{equation}
We say that an absolutely continuous curve
$\gamma:[a,b]\to\Omega$ is {\em horizontal} if there are measurable functions
$a_i(t)$, $a\leq t\leq b$, $i=1,2,\ldots,m$ such that
$$
\gamma'(t)=\sum_{i=1}^m a_i(t)X_i(\gamma(t))
\quad
\mbox{for almost all $t\in [a,b]$.}
$$
The horizontal length of $\gamma$ is defined as
$$
\ell_H(\gamma)=\int_a^b |\gamma'(t)|_H\, dt.
$$
Denoting the Euclidean length of a curve $\gamma$ by $\ell(\gamma)$,
it easily follows from \eqref{2341} that
if $G\Subset\Omega$, then there is a constant $C\geq 1$ such that for any
horizontal curve $\gamma:[a,b]\to G$ we have
\begin{equation}
\label{pppi}
C^{-1}\ell(\gamma)< \ell_{H}(\gamma)\leq C\ell(\gamma).
\end{equation}
Assume that any two points in $\Omega$ can be connected by a horizontal curve.
This is the case for example if the vector fields satisfy the H\"ormander condition
\cite[Proposition~III.4.1]{VSC}. 
All the assumptions about the vector fields given above are satisfied by 
Carnot groups (and in particular by the Heisenberg groups), \cite[Section~11.3]{SMP},
but not by the Grushin type spaces \cite{FGW}. Namely in general in the Grushin type spaces the inequality 
$\ell_{H}(\gamma)\leq C\ell(\gamma)$ need not be satisfied.

The Carnot-Carath\'eodory distance $d_{cc}(x,y)$ of the points $x,y\in\Omega$
is defined as the infimum of horizontal lengths of 
horizontal curves connecting $x$ and $y$. Since we assume that any two points in $\Omega$
can be connected by a horizontal curve, $(\Omega,d_{cc})$ is a metric space.

Clearly horizontal curves are rectifiable and it is well-known that
every rectifiable curve with the arc-length parametrization is horizontal.
Moreover $\ell_H(\gamma)$ equals the length $\ell_{cc}(\gamma)$ of $\gamma$ with respect to the 
Carnot-Carath\'eodory metric. A detailed account on this topic can be found in \cite{monti}.
Hence \eqref{pppi} implies that the mapping
${\rm id}:(\Omega,d_{cc})\to\Omega$ from the Carnot-Carath\'eodory space onto $\Omega$ with Euclidean metric
is locally weakly BLD.

The next result follows immediately from a local version of Theorem~\ref{main-quasi}.
It applies to Carnot groups and in particular to the Heisenberg groups. 
\begin{theorem}
\label{main-v}
Let $X_1,\ldots,X_m$ be a family of locally Lipschitz vector fields in an open and connected domain $\Omega\subset\bbbr^n$ such that
for every compact set $K\subset\Omega$ 
\begin{equation}
\label{e2}
\inf_{p\in K} \inf_{i\in \{1,\ldots,m\}}|X_i(p)|>0.
\end{equation}
Assume also that any two points in $\Omega$ can be connected by a horizontal curve. 
Then for $k\geq 1$ and any Lipschitz mapping $f:\bbbr^k\supset E\to (\Omega,d_{cc})$
the following conditions are equivalent.
\begin{enumerate}
\item $\H^k_{d_{cc}}(f(E))=0$ in $(\Omega,d_{cc})$;
\item $\H^k(f(E))=0$ with respect to the Euclidean metric in $\Omega$;
\item $\rank (\ap Df)<k$ a.e. in $E$.
\end{enumerate}
\end{theorem}
Let us briefly describe how this result applies to Carnot groups. 
For more details, see \cite{magnani}.
If $G$ is a Carnot group and the first layer of the stratification of 
the Lie algebra $\mathfrak{g}$ does not contain a $k$-dimensional Lie subalgebra, then
it follows from the Pansu differentiability theorem that the rank of the approximate derivative
of any Lipschitz mapping $f:\bbbr^k\supset E\to G$ is less than $k$ a.e., so
$\H^k_{d_{cc}}(f(E))=0$ by Theorem~\ref{main-v}.
Hence $G$ is purely $k$-unrectifiable. 
This slightly simplifies the proof of Theorem~1.1 in \cite{magnani}.


\begin{thebibliography}{888}
%
\bibitem{ambrosiok}
{\sc Ambrosio, L., Kirchheim, B.:}
Rectifiable sets in metric and Banach spaces. 
{\em Math.\ Ann.} 318 (2000), 527--555. 
%
\bibitem{ambrosiok2}
{\sc Ambrosio, L., Kirchheim, B.:}
Currents in metric spaces. 
{\em Acta Math.} 185 (2000),1--80. 
%
\bibitem{BHW}
{\sc Balogh, Z. M., Haj\l{}asz, P., Wildrick, K.:} Weak contact equations for mappings into Heisenberg groups.
{\em Indiana Univ.\ Math. J.} (to appear).
%
\bibitem{DS}
{\sc David, G., Semmes, S.:}
{\em Fractured fractals and broken dreams. Self-similar geometry through metric and measure.} 
Oxford Lecture Series in Mathematics and its Applications, 7. 
The Clarendon Press, Oxford University Press, New York, 1997.
%
\bibitem{di}
{\sc DiBenedetto, E.:}
{\em Real analysis.}
Birkh\"auser Advanced Texts: Basler Lehrb\"ucher. 
[Birkh\"auser Advanced Texts: Basel Textbooks] 
Birkh\"auser Boston, Inc., Boston, MA, 2002.
%
\bibitem{EG}
{\sc Evans, L. C., Gariepy, R. F.:} 
{\em Measure theory and fine properties of functions.} 
Studies in Advanced Mathematics. CRC Press, Boca Raton, FL, 1992. 
%
\bibitem{federer}
{\sc Federer, H.:}
{\em Geometric measure theory}. 
Die Grundlehren der mathematischen Wissenschaften, 
Band 153 Springer-Verlag New York Inc., New York 1969 
%
\bibitem{FGW}
{\sc Franchi, B., Guti\'errez, C. E., Wheeden, R. L.:} 
Weighted Sobolev-Poincar\'e inequalities for Grushin type operators. 
{\em Comm.\ Partial Differential Equations} 19 (1994), 523--604.
%
\bibitem{gromov}
{\sc Gromov, M.:}
{\em Partial differential relations.} 
Ergebnisse der Mathematik und ihrer Grenzgebiete (3) 
[Results in Mathematics and Related Areas (3)], 9. 
Springer-Verlag, Berlin, 1986.
%
\bibitem{hajlasz2}
{\sc Haj\l{}asz, P.:}
Change of variables formula under minimal assumptions. 
{\em Colloq.\ Math.} 64 (1993), 93--101. 
%
\bibitem{SMP}
{\sc Haj\l{}asz, P., Koskela, P.:}
Sobolev met Poincar\'e. 
{\em Mem.\ Amer.\ Math.\ Soc.} 145 (2000), no. 688, x+101 pp.
%
\bibitem{heinonen}
{\sc Heinonen, J.:}
{\em Lectures on analysis on metric spaces.} 
Universitext. Springer-Verlag, New York, 2001.
%
\bibitem{kirchheim}
{\sc Kirchheim, B.:}
Rectifiable metric spaces: local structure and regularity of the Hausdorff measure. 
{\em Proc.\ Amer.\ Math.\ Soc.} 121 (1994), 113--123.
%
\bibitem{ledonne}
{\sc Le Donne, E.:}
Lipschitz and path isometric embeddings of metric spaces. 
{\em Geom.\ Dedicata} 166 (2013), 47--66. 
%
\bibitem{magnani}
{\sc Magnani, V.:} Unrectifiability and rigidity in stratified groups. 
{\em Arch.\ Math.\ (Basel)} 83 (2004), 568--576.
%
\bibitem{MalyZ}
{\sc Mal\'y, J., Ziemer, W. P.:}
{\em Fine regularity of solutions of elliptic partial differential equations.} 
Mathematical Surveys and Monographs, 51. American Mathematical Society, Providence, RI, 1997.
%
\bibitem{martiov}
{\sc Martio, O., V\"ais\"al\"a, J.:} 
Elliptic equations and maps of bounded length distortion. 
{\em Math.\ Ann.} 282 (1988), 423--443.
%
\bibitem{Mattila}
{\sc Mattila, P.:}
{\em Geometry of sets and measures in {E}uclidean spaces.} 
Cambridge Studies in Advanced Mathematics, Vol.\ 44.
Cambridge University Press, Cambridge, 1995. 
%
\bibitem{monti}
{\sc Monti R.:} 
{\em Distances, boundaries and surface measures in Carnot-Carath\'eodory spaces}, PhD thesis 2001.
%
\bibitem{sternberg}
{\sc Sternberg, S.:}
{\em Lectures on differential geometry.} 
Second edition. With an appendix by Sternberg and Victor W. Guillemin. 
Chelsea Publishing Co., New York, 1983. 
%
\bibitem{VSC}
{\sc Varopoulos, N. Th., Saloff-Coste, L., Coulhon, T.:}
{\em Analysis and geometry on groups.}
Cambridge Tracts in Mathematics, 100. Cambridge University Press, Cambridge, 1992.
%
\bibitem{whitney}
{\sc Whitney, H.:} 
On totally differentiable and smooth functions. 
{\em Pacific J. Math.} 1 (1951), 143--159.
%
\end{thebibliography}
\end{document}